\documentclass{jloganal}

\usepackage{mathrsfs}
\usepackage{amsfonts}
\usepackage{amsbsy}
\usepackage{bm}
\def\boldfacedelta{\boldsymbol{\Delta}_2^0}

\title{A Cantor-Bendixson-like process which detects $\boldfacedelta$}

\author{Samuel Alexander}
\givenname{Samuel}
\surname{Alexander}
\address{231 W.~18th Avenue, Columbus, OH 43210}
\email{alexander@math.ohio-state.edu}
\urladdr{http://www.math.osu.edu/~alexander}


\volumenumber{}
\issuenumber{}
\publicationyear{}
\papernumber{}
\startpage{}
\endpage{}
\doi{}
\MR{}
\Zbl{}
\received{}
\revised{}
\accepted{}
\published{}
\publishedonline{}
\proposed{}
\seconded{}
\corresponding{}
\editor{}
\version{}

\theoremstyle{plain}
\newtheorem{definition}{Definition}
\newtheorem{theorem}[definition]{Theorem}
\newtheorem{lemma}[definition]{Lemma}

\newtheorem{proposition}[definition]{Proposition}

\def\N{\mathbb{N}}
\def\rest{\restriction}

\begin{document}

\section{Addendum}

\begin{abstract}

For each subset of Baire space, we define, in a way similar to
a common proof of the Cantor-Bendixson Theorem, a sequence of decreasing subsets $S_{\alpha}$
of $\N^{<\N}$, indexed by ordinals.  We use this to obtain two new characterizations of the boldface
$\boldfacedelta$ Borel pointclass.  ADDENDUM: In January 2012 we learned that the notion of
guessability appeared in an equivalent form, and even with the same name, in the doctoral dissertation
of William Wadge \cite{wadge}.  As for the main result of this paper, Wadge proved one direction
and gave a proof for the other direction which he attributed to Hausdorff.  The proofs in this paper
present an alternate means to those results.
\end{abstract}

\maketitle

Please read the addendum in the above abstract for an important note on this paper's unoriginality.

The usual Cantor-Bendixson derivative ``detects'' countability, in the sense
that the perfect kernel of $S\subseteq\N^{\N}$ (the result of applying the Cantor-Bendixson derivative
repeatedly until a fixed point is reached) is empty if and only if $S$ is countable (\cite{kechris}, page 34).
In this paper, I will show a process which ``detects'' $\boldfacedelta$:
a process which depends on $S\subseteq\N^{\N}$ and which reaches a fixed point or kernel,
a kernel which will be empty if and only if $S$ is $\boldfacedelta$.

\begin{definition}
Suppose $S\subseteq\N^{\N}$.  If $X\subseteq\N^{<\N}$, let $[X]$ denote the set of infinite sequences whose
initial segments are all in $X$.
\begin{itemize}
\item
Define $S_{\alpha}\subseteq \N^{<\N}$ for every ordinal $\alpha$ by induction as follows:
$S_0=\N^{<\N}$,
$S_{\lambda}=\cap_{\beta<\lambda}S_{\beta}$ for any limit ordinal $\lambda$.
And finally, for any ordinal $\beta$ define
\[
S_{\beta+1}
=
\{
 x\in S_{\beta}\, :\, \mbox{$\exists x',x''\in [S_{\beta}]$ such that $x\subseteq x'$, $x\subseteq x''$, $x'\in S$, and $x''\not\in S$}
\}
\]
\item
Let $\alpha(S)$ be the minimal ordinal $\alpha$ such that $S_{\alpha}=S_{\alpha+1}$.
\item
Let $S_{\infty}=S_{\alpha(S)}$ (the \emph{kernel} of the above process).
\end{itemize}
\end{definition}

Throughout the paper, $S$ will denote a subset of $\N^{\N}$.
If $f:\N\to\N$ and $n\in\N$, I will use $f\rest n$ to denote $(f(0),\ldots,f(n-1))$; $f\rest 0$ will denote
the empty sequence.
My goal is to prove that the following are equivalent:
\begin{itemize}
\item $S$ is $\boldfacedelta$.
\item $S_{\infty}=\emptyset$.
\item $S=T\backslash [T_{\infty}]$ for some $T$.
\end{itemize}

The reader might wonder why I define $S_{\alpha}$ to lie in $\N^{<\N}$, rather than in $\N^{\N}$ as one might
expect by extrapolating from the classical Cantor-Bendixson derivative.  Why not define a new derivative
\[S^{*}=\{f\in S\,:\,\mbox{$f$ is a limit of points of $S\backslash\{f\}$ and also of $S^c$}\},\] and 
then follow
Cantor-Bendixson more directly?  If we do this, we end up getting a kernel which does \emph{not}
detect $\boldfacedelta$.  For example, let $S=\{f\,:\,\forall n\,f(n)\not=0\}$, a $\boldfacedelta$
subset of $\N^{\N}$.  It's easy to see $S^{*}=S$, whereas in order for our process to detect $\boldfacedelta$,
we would like for it to reduce $S$ to $\emptyset$.
The reader can check that $S_0=\N^{<\N}$, $S_1$ is the set of finite sequences
not containing $0$, and $S_2=\emptyset$.

\begin{definition}
\label{defofguessable}
Say that $S\subseteq\N^{\N}$ is  \emph{guessable} if there
is a function $G:\N^{<\N}\to\N$ such that for every $f:\N\to\N$,
\[
\lim_{n\to\infty} G(f\rest n) = \left\{\begin{array}{l}\mbox{$1$, if $f\in S$;}\\ 
\mbox{$0$, otherwise.}\end{array}\right.
\]
If so, we say $G$ is a \emph{guesser} for $S$.
\end{definition}

\begin{theorem}
\label{guesscharacterize}
A subset of $\N^{\N}$ is guessable if and only if it is $\boldfacedelta$.
\end{theorem}

This theorem is proved on page 11 of Alexander \cite{alexanderjis}.  It is also a special case of the main
theorem of Alexander \cite{alexanderpreprint}.

\begin{proposition}
\label{piece1}
Suppose $S$ is $\boldfacedelta$.  Then $S_{\infty}=\emptyset$.
\end{proposition}

\begin{proof}
Contrapositively, suppose $S_{\infty}\not=\emptyset$.
I will show $S$ is non-guessable,
hence non-$\boldfacedelta$ by Theorem~\ref{guesscharacterize}.
Assume not, and let $G:\N^{<\N}\to\N$ be a guesser for $S$.
I will build a sequence 
on whose initial segments
$G$ diverges,
contrary to Definition~\ref{defofguessable}.
There is some $\sigma_0\in S_{\infty}$.
Now inductively suppose I've defined finite sequences
$\sigma_0\subset_{\not=} \cdots \subset_{\not=} \sigma_k$
in $S_{\infty}$
such that for $0<i\leq k$,
$G(\sigma_i)\equiv i\mbox{ mod 2}$.
Since $\sigma_k\in S_{\infty}=S_{\alpha(S)}=S_{\alpha(S)+1}$,
this means there are $\sigma',\sigma''\in [S_{\infty}]$,
extending $\sigma_k$, with $\sigma'\in S$, $\sigma''\not\in S$.
Choose $\sigma\in\{\sigma',\sigma''\}$
with $\sigma\in S$ iff $k$ is even.
Then $\lim_{n\to\infty}G(\sigma\rest n)\equiv k+1\mbox{ mod 2}$.
Let $\sigma_{k+1}\subset\sigma$ properly extend $\sigma_k$ such that $G(\sigma_{k+1})\equiv k+1\mbox{ mod 2}$.
Note $\sigma_{k+1}\in S_{\infty}$ since $\sigma\in [S_{\infty}]$.

By induction, I've defined $\sigma_0\subset_{\not=}\sigma_1\subset_{\not=}\cdots$ such
that for $i>0$, $G(\sigma_i)\equiv i\mbox{ mod 2}$.  This contradicts Definition~\ref{defofguessable}
since $\lim_{n\to\infty}G((\cup_i\sigma_i)\rest n)$ ought to converge.
\end{proof}

For the converse we need more machinery.

\begin{definition}
If $\sigma\in\N^{<\N}$, $\sigma\not\in S_{\infty}$, then let $\beta(\sigma)$
denote the least ordinal such that $\sigma\not\in S_{\beta(\sigma)}$.
\end{definition}

Note that whenever $\sigma\not\in S_{\infty}$, $\beta(\sigma)$ is a successor ordinal.

\begin{lemma}
\label{techlemma1}
Suppose $\sigma\subseteq\tau$ are finite sequences.  If $\tau\in S_{\infty}$
then $\sigma\in S_{\infty}$.  And if $\sigma\not\in S_{\infty}$, then $\beta(\tau)\leq 
\beta(\sigma)$.
\end{lemma}

\begin{proof}
It's enough to show for any ordinal $\beta$ if $\tau\in S_{\beta}$ then $\sigma\in S_{\beta}$.
This is by induction on $\beta$, the limit case and $\beta=0$ case being trivial.
Assume $\beta$ is successor.
If $\tau\in S_{\beta}$, this means $\tau\in S_{\beta-1}$ and there are $\tau',\tau''\in [S_{\beta-1}]$ extending $\tau$ with $\tau'\in S$,
$\tau''\not\in S$.
Since $\tau'$ and $\tau''$ extend $\tau$, and $\tau$ extends $\sigma$, $\tau'$ and $\tau''$ extend $\sigma$,
and since $\sigma\in S_{\beta-1}$ by induction, this shows $\sigma\in S_{\beta}$.
\end{proof}

\begin{lemma}
\label{techlemmaX}
Suppose $f:\N\to\N$, $f\not\in [S_{\infty}]$.
Then there is some $i$ such that for all $j\geq i$,
$f\rest j\not\in S_{\infty}$ and $\beta(f\rest j)=\beta(f\rest i)$.
Furthermore, $f\in [S_{\beta(f\rest i)-1}]$.
\end{lemma}

\begin{proof}
The first part of the lemma follows from Lemma~\ref{techlemma1} and the well-foundedness
of $ORD$.
For the second part we must show $f\rest k\in S_{\beta(f\rest i)-1}$ for every $k$.
If $k\leq i$, then $f\rest k\in S_{\beta(f\rest i)-1}$ by Lemma~\ref{techlemma1}.
If $k\geq i$, then $\beta(f\rest k)=\beta(f\rest i)$ and so $f\rest k\in S_{\beta(f\rest i)-1}$
since it is in $S_{\beta(f\rest k)-1}$ by definition of $\beta$.
\end{proof}

\begin{proposition}
\label{piece2}
If $S_{\infty}=\emptyset$ then $S$ is $\boldfacedelta$.
\end{proposition}

\begin{proof}
Assume $S_{\infty}=\emptyset$.  I will define a function $G:\N^{<\N}\to\N$ which guesses $S$,
which is sufficient by Theorem~\ref{guesscharacterize}.

Let $\sigma\in\N^{<\N}$, we have $\sigma\not\in S_{\infty}$ (since $S_{\infty}=\emptyset$)
and so $\sigma\in S_{\beta(\sigma)-1}\backslash S_{\beta(\sigma)}$.
Since $\sigma\not\in S_{\beta(\sigma)}$,
this means that for every two extensions $\sigma',\sigma''$ of $\sigma$ in
$[S_{\beta(\sigma)-1}]$, either $\sigma'$ and $\sigma''$ are both in $S$,
or both are outside $S$.
It might be that there \emph{is} no extension of $\sigma$ lying in
$[S_{\beta(\sigma)-1}]$.  In that case, arbitrarily define $G(\sigma)=0$.
But if there are such extensions, let $G(\sigma)=1$ if all of those extensions are in
$S$, and let $G(\sigma)=0$ if all of those extensions are outside $S$.

I claim $G$ guesses $S$.  To see this, let $f\in S$.
I will show $G(f\rest n)\rightarrow 1$ as $n\rightarrow\infty$.
Since $f\not\in [S_{\infty}]$, let $i$ be as in Lemma~\ref{techlemmaX}.
I claim $G(f\rest j)=1$ whenever $j\geq i$.
Fix $j\geq i$.
We have $\beta(f\rest j)=\beta(f\rest i)$ by choice of $i$,
and $f\in [S_{\beta(f\rest i)-1}]=[S_{\beta(f\rest j)-1}]$.
By the previous paragraph, if any infinite sequence extends $f\rest j$
and lies in $[S_{\beta(f\rest j)-1}]$, then either all such sequences
are in $S$, or all are outside $S$.  One such sequence is $f$, and it is inside $S$,
and therefore, all such sequences are inside $S$, whereby $G(f\rest j)=1$ as desired.

Identical reasoning shows that if $f\not\in S$ then $\lim_{n\to\infty}G(f\rest n)=0$.
So $G$ guesses $S$, $S$ is guessable, and by Theorem~\ref{guesscharacterize}, $S$ is $\boldfacedelta$.
\end{proof}

\begin{theorem}
\label{maintheorem}
$S$ is $\boldfacedelta$ if and only if $S_{\infty}=\emptyset$.
\end{theorem}

\begin{proof}
By combining Propositions~\ref{piece1} and \ref{piece2}.
\end{proof}

We will close by giving one more characterization of $\boldfacedelta$.

\begin{theorem}
$S$ is $\boldfacedelta$ if and only if $S=T\backslash [T_{\infty}]$ for some $T\subseteq\N^{\N}$.
\end{theorem}

\begin{proof}
By Theorem~\ref{maintheorem}, if $S$ is $\boldfacedelta$ then $S=S\backslash [S_{\infty}]$.
For the converse, it suffices to let $S$ be arbitrary and prove $S\backslash [S_{\infty}]$ is $\boldfacedelta$.

By Theorem~\ref{guesscharacterize},
it is enough to exhibit a guesser $G:\N^{<\N}\to\N$ for $S\backslash [S_{\infty}]$.
Let $\sigma\in\N^{<\N}$.
If $\sigma\in S_{\infty}$, let $G(\sigma)=0$.
Otherwise, if $\sigma$ has at least one infinite extension in $[S_{\beta(\sigma)-1}]$, and all such extensions are also in $S$,
then let $G(\sigma)=1$.
In any other case, let $G(\sigma)=0$.

We claim $G$ guesses $S\backslash [S_{\infty}]$.

Case 1: $f\in S\backslash [S_{\infty}]$.
By Lemma~\ref{techlemmaX}, find an $i$ such that for all $j\geq i$ we have $f\rest j\not\in S_{\infty}$ and
$\beta(f\rest j)=\beta(f\rest i)$
and
$f\in [S_{\beta(f\rest i)-1}]$.
Thus for any $j\geq i$, $f\rest j$ does have one extension in $[S_{\beta(f\rest j)-1}]$,
namely $f$ itself, and $f$ is in $S$.  All other such extensions must also be
in $S$, or else we would have $f\rest j\in S_{\beta(f\rest j)}$, violating the definition of $\beta$.
So $G(f\rest j)=1$, showing that $G(f\rest n)\to 1$ as $n\to\infty$.

Case 2: $f\in [S_{\infty}]$.
Then for every $n$, $f\rest n\in S_{\infty}$ and thus by definition $G(f\rest n)=0$.

Case 3:
$f\not\in S$ and $f\not\in [S_{\infty}]$.
As in Case 1, find $i$ such that for all $j\geq i$,
$f\rest j\not\in S_{\infty}$ and $\beta(f\rest j)=\beta(f\rest i)$,
and $f\in [S_{\beta(f\rest j)-1}]$.
For any $j\geq i$, $f\rest j$ has one extension in $[S_{\beta(f\rest j)-1}]$,
namely $f$, and $f\not\in S$; so by definition $G(f\rest j)=0$.
\end{proof}

\end{document}